\documentclass[runningheads]{llncs}
\usepackage[T1]{fontenc}
\usepackage{color}

\usepackage{amsmath,amssymb}
\usepackage{centernot}
\usepackage{hyperref}

\usepackage{tikz}

\newcommand{\R}{\mathbb{R}}

\newcommand{\Z}{\mathbb{Z}}
\newcommand{\N}{\mathbb{N}}
\newcommand{\ID}{\mathrm{id}}

\begin{document}
\title{Symbol Frequencies in Surjective Cellular Automata}

% TODO: institutions
\author{Benjamin Hellouin de Menibus\inst{1}\orcidID{0000-0001-5194-929X} \and Ville Salo\inst{2}\orcidID{0000-0002-2059-194X} \and Ilkka Törmä\inst{2}\orcidID{0000-0001-5541-8517}}

\institute{
Université Paris-Saclay, CNRS, Laboratoire Interdisciplinaire des Sciences du Numérique, Orsay, France,
\email{hellouin@lisn.fr}
\and
Department of Mathematics and Statistics, University of Turku, Turku, Finland, \email{\{vosalo,iatorm\}@utu.fi}}

\maketitle

\begin{abstract}
  We study the behavior of probability measures under iteration of a surjective cellular automaton.
  We solve the following question in the negative: if the initial measure is ergodic and has full support, do all weak-* limit points of the sequence of measures have full support as well?
  The initial measure of our solution is not a product measure, and in this case the question remains open.
  To this end, we present a tool for studying the frequencies of symbols in preimages of surjective cellular automata, and prove some basic results about it. % do we know they are nontrivial? :P
  However, we show that by itself it is not enough to solve the stricter question in the positive.
\end{abstract}

% TODO: more keywords

\section{Introduction}

We are interested in the asymptotic behavior of cellular automata (CA) iterated on a random initial configuration. Formally, we consider a CA $F$ on a state set $\Sigma$ and a shift invariant probability measure $\mu$ on $\Sigma^\Z$ -- such as the uniform measure $\lambda$ where each cell is chosen in an i.i.d.\ uniform manner -- and study the asymptotic properties of the sequence $(F^t \mu)_{t \in \N}$. This asymptotic behavior may be computationally very complex \cite{hellouin2013} and related properties are generally undecidable \cite{BoDePoSaTh15}. 

Surjective CA are better-behaved in this regard. Despite the fact that they are capable of universal computation in some settings, their behavior on probability measures has a much more combinatorial and dynamical structure. For example, every surjective CA keeps the uniform measure invariant, so in particular properties of $(F^t \lambda)_{t \in \N}$ are completely understood.
Another very important property is that the image of a measure under a surjective CA cannot have a lower measure-theoretical entropy. A survey of this structure can be found in \cite{kari2015}.

Let us provide examples of asymptotic behavior. It may happen that $F$ has a non-uniform invariant measure $\mu$, in which case $F^t \mu = \mu$ stays ``away'' from the uniform measure; this happens e.g. when $F = \ID$, or more generally when the CA has some conserved quantity \cite{KaTa12}. 
Other types of CA may exhibit \emph{randomization}, where for a large class of initial measures the sequence $(F^t \mu)_{t \in \N}$ tends, either directly or on average, toward the uniform measure.
As the latter is the measure of maximum entropy on $\Sigma^\Z$, the intuition is that as time goes on, $F^t \mu$ tends to a measure of maximum disorder.
Measure randomization has been studied extensively (see \cite{pivato2002,hellouin2020} among others).

Apart from specific classes of initial measures (uniform measures, Gibbs measures \cite{kari2015}), many questions remain open about the behavior of general surjective CA. These two families, as well as the fact that surjective CA cannot decrease entropy, are examples of phenomena where $(F^t \mu)_{t \in \N}$ cannot go ``further away'' from the uniform measure and that the identity CA is extremal in this sense. In this article, we provide new examples and counterexamples of this phenomenon.

Our main motivation is the following question, which is hinted at in~\cite[Section~7]{BoDePoSaTh15} and relayed to us explicitly by Guillaume Theyssier (personal communication). For a probability measure $\mu$ and a finite word $u$, we denote by $\mu([u])$ the probability that $u$ appears at position $0$.

\begin{question}
\label{q:FullSupportLimits}
Let $F$ be a surjective CA on $\Sigma^\Z$ and $\mu$ a product measure of full support, meaning $\mu([u]) > 0$ for all words $u$. Does $F^t\mu([u]) \centernot\longrightarrow 0$ hold for all $u$?
\end{question}

%\begin{question}
%\label{q:FullSupportLimits}
%Let $F$ be a surjective CA on $\Sigma^\Z$.
%Is it necessarily true that for all full support product measures $\mu$ on $\Sigma^\Z$, every weak limit point of the sequence of measures $(F^n \mu)_{n \in \N}$ has full support?
%\end{question}

This article represents an attempt at resolving this question that ultimately failed, but produced interesting partial results and counterexamples. For example, the above intuition of not going further away from the uniform measure may suggest a stronger statement, that is, 
\begin{equation}\forall u\in\Sigma^n, |F\mu([u]) - \lambda([u])|\leq \max_{w\in\Sigma^n}|\mu([w]) - \lambda([w])|,\label{eq:faruniform}\end{equation}
where $\lambda([u]) = |\Sigma|^{-n}$.
We show that this is false (Theorem~\ref{thm:NotFullSupport}).
However, a failed attempt led us to introduce a value called \emph{correlation} for which the phenomenon holds (Theorems \ref{thm:OneDomination} and~\ref{thm:HighDomination}) and which does not seem to be a consequence of existing results, although we did not find a clear application.

Orthogonally to these results, we also show that Question~\ref{q:FullSupportLimits} is false if the measure $\mu$ is merely assumed to be ergodic instead of a product measure.
We construct an explicit counterexample in which the two-neighbor XOR automaton $F$ satisfies $F^t([1]) \longrightarrow 0$ (Theorem~\ref{thm:NotFullSupportXOR}).
We also provide an example of the XOR automaton giving a different limit behavior (Theorem~\ref{thm:XORall1}).

%For all these questions it is enough to consider the case of words of length $1$ by grouping blocks of symbols together, increasing the size of the alphabet.
%Thus, our approach was to study the frequencies of single symbols in the local rule of a surjective CA.
%This lead us to define the \emph{$(a,b)$-correlation} of a local rule, which measures the number of occurrences of a symbol $a$ in all preimages of another symbol $b$.

%Our main results are that the identity CA has the highest $(a,a)$-correlation among all surjective CA of a given radius (Theorem~\ref{thm:OneDomination} and Theorem~\ref{thm:HighDomination}), and that correlation gives a necessary condition for conserving a given symbol (Theorem~\ref{thm:Conserving}).
%We also show that correlation alone is not enough to provide a positive answer to Question~\ref{q:FullSupportLimits}, since it does not hold for arbitrary sequences of surjective CA in place of $(F^n)_{n \in \N}$ (Theorem~\ref{thm:NotFullSupport}).

\section{Definitions}

% Notation; some of it is ugly, change?
% Benjamin: today, I would rather give the same name to all three versions of the CA. also I don't think writing N_r instead of r+1 is clearer.
\paragraph*{Cellular automata}
We consider one-dimensional cellular automata over a finite alphabet $\Sigma$.
A \emph{cellular automaton} (CA for short) is a function $F : \Sigma^\Z \to \Sigma^\Z$ defined by a \emph{local rule} $f : \Sigma^{1+r} \to \Sigma$ of some \emph{radius} $r \geq 0$ by $F(x)_i = f(x_{[i,i+r]})$ for all $x \in \Sigma^\Z$ and $i \in \Z$.
We may apply the CA $F$ to a word $w \in \Sigma^n$ of length $n \geq r+1$ to obtain a word $F(w) \in \Sigma^{n-r}$ by the formula $F(w)_i = f(w_{[i, i+r]})$.

The radius-$0$ identity CA is defined by $\ID(x) = x_0$. The radius-$1$ shift CA $\sigma$ is defined by $\sigma(x) = x_1$.
Note that we only consider cellular automata whose neighborhood extends to the right.
This is because the notions we study are invariant under composition with shift maps, and one-sided neighborhoods make our proofs somewhat less technical.

For a word $w \in \Sigma^*$ and $A \subset \Sigma$, we denote by $|w|$ its length and by $|w|_A$ the number of occurrences of $A$-symbols in $w$.
For single symbols $a \in \Sigma$, we write $|w|_a = |w|_{\{a\}}$. The \emph{cylinder} corresponding to $w$, denoted $[w]$, is the set of configurations $x\in\Sigma^\Z$ such that $x_{[0,|w|-1]} = w$.

\paragraph*{Surjectivity and probability measures}
Through the article, probability measures on $\Sigma^\Z$ are assumed to be $\sigma$-invariant. In this case a probability measure $\mu$ is entirely described by its value on cylinders $\mu([u])$, $u\in\Sigma^\ast$.

The simplest example is the uniform measure $\lambda$ where $\lambda([u]) = |\Sigma|^{-|u|}$.
In general, any function $\pi : \Sigma \to [0,1]$ with $\sum_{a \in \Sigma} \pi(a) = 1$ defines a product measure $\mu$ with $\mu([u]) = \prod_{i = 0}^{|u|-1} \pi(u_i)$.
When $\Sigma =  \{0,1\}$, the product measure with a density $p$ of $1$-symbols is denoted by $\mu_p$. It is also called the Bernoulli measure with parameter $p$.
For $a \in \Sigma$, the Dirac measure with support ${}^\infty a^\infty$, denoted $\delta_a$, is defined by $\delta_a([a^n]) = 1$ and $\delta_a([u]) = 0$ for all $u \in \Sigma^\ast \setminus a^\ast$.

%Would you say that the product measure on other alphabets are called Bernoulli as well? What should we call them if not?
%ILKKA: a referee once complained to me that Bernoulli measures are specifically binary, and other product measures are just "product measures".

The image under a CA $F$ of a probability measure $\mu$ is given by $F\mu([u]) = \mu(F^{-1}([u]))$ for any $u\in\Sigma^\ast$. We call $\nu$ the weak-* limit of the sequence $(F^t\mu)_{t\in\N}$ if $F^t\mu([u]) \longrightarrow \nu([u])$ for all words $u\in\Sigma^\ast$. The sequence $(F^t\mu)_{t\in\N}$ does not always admit a single weak-* limit, but it has limit points by compactness.

A surjective CA $F$ satisfies a \emph{balance property}: $|F^{-1}(w)| = |\Sigma|^r$ holds for all words $w \in \Sigma^\ast$. This implies that $F\lambda = \lambda$.

The measure-theoretical entropy of a shift-invariant probability measure $\mu$ is
\[ h(\mu) = - \lim_{n\to\infty} \sum_{w\in\Sigma^n} \mu([w])\log\mu([w]).\]
In particular, we have $h(\lambda) = \log(|\Sigma|)$, which is the highest possible entropy.
For any function $\pi : \Sigma \to [0,1]$ with $\sum_{a \in \Sigma} \pi(a) = 1$, the product measure defined by $\pi$ is the unique measure of maximum entropy among those measures $\mu$ with $\mu([a]) = \pi(a)$ for all $a \in \Sigma$.
For any CA $F$, entropy may not increase: $h(F\mu) \leq h(\mu)$.
If $F$ is surjective, it must preserve entropy: $h(F \mu) = h(\mu)$.

\section{Limit Measures not of Full Support}

We show that if one considers an arbitrary sequence of surjective CA, instead of the iterations of a single CA, the answer to Question~\ref{q:FullSupportLimits} is negative.
% TODO: explain how this shows that correlation is not enough to answer the question

\begin{theorem}
  \label{thm:NotFullSupport}
  There exists $0 < p < 1$, a sequence $(F_n)_{n \in \N}$ of reversible binary cellular automata and a word $w\in\{0,1\}^\ast$ such that $F_n([w]) \longrightarrow 0$. In particular, no weak-* limit point of $(F_n \mu_p)_{n \in \N}$ has full support.
\end{theorem}

\begin{proof}
  Take a small $p > 0$, and for $n \in \N$, let $w_n = (10)^n 0$ and $\alpha(n) = n/\mu_p(w_n) + |w_n|$.
  For each $k \geq 2/p$, define two subsets of $\{0,1\}^k$:
  \[
    A_k = \{ w \in \{0,1\}^k \mid kp/2 \leq |w|_1 \leq 3kp/2, w_n \text{~does not occur in~} w \}
  \]
\[B_k = \{ (110 a_1) (110 a_2) \cdots (110 a_{\lfloor k/4 \rfloor}) 0^{k - 4\lfloor k/4 \rfloor} \mid a_i \in \{0,1\} \}\]
  We also fix an injection $E_k : A_k \to B_k$.
  This is possible, since
  \[
    |A_k| \leq \sum_{m = pk/2}^{3pk/2} \binom{k}{m} \leq pk \binom{k}{3pk/2} \leq pk \cdot 2^{k H(3p/2)} \leq 2^{\lfloor k/4 \rfloor} = |B_k|
  \]
  as long as $p$ is small enough.
  Finally, we have $A_k \cap B_k = \emptyset$, because every $w \in B_k$ satisfies $|w|_1 \geq 2 \lfloor k/4 \rfloor \geq k/3 > 3kp/2$ as long as $p$ is small enough.
  
  We are now ready to define the CA $F_n$.
  It behaves as follows on a configuration $x \in \{0,1\}^\Z$.
  Let $I = \{ i \in \Z \mid x_{[i,i+|w_n|-1]} = w_n\}$ be the set of occurrences of $w_n$.
  For any two consecutive elements $i_1 < i_2 \in I$, the word $w_n v = x_{[i_1, i_2-1]}$ is called an \emph{$n$-interval}.
  Denote $k = i_2 - i_1$.
  This $n$-interval is \emph{short} if $k < 2n/p$, \emph{medium} if $2n/p \leq k \leq \alpha(n) + |w_n|$, and \emph{long} if $k > \alpha(n) + |w_n|$.
  \begin{itemize}
  \item If an interval is medium and $v \in A_k$, then $F_n$ replaces $v$ by $E_k(v) \in B_k$.
  Such intervals are called \emph{common}.
  \item If an interval is medium and $v \in E_k(A_k)$, then $F_n$ replaces $v$ by $E_k^{-1}(v) \in A_k$. %; such intervals are called \emph{backward active}.
  \item $F_n$ acts as the identity in any other situation.
  \end{itemize}
  Then $F_n$ is a well-defined CA with two-sided radius $\alpha(n) + |w_n|$.
  It cannot create or destroy occurrences of $w_n$.
  It satisfies $F_n^{-1} = F_n$, and is in particular reversible.

  We claim that $w=1111$ does not occur in the support of any weak-* limit point of $F_n \mu_p$, or equivalently, $(F_n \mu_p)([1111]) \longrightarrow 0$ as $n \longrightarrow \infty$.
  Namely, consider a configuration $x \in \{0,1\}^\Z$ with $F_n(x) \in [1111]$.
  Almost surely with respect to $\mu_p$, $x$ partitions $\Z$ into $n$-intervals of finite length.
  In any occurrence of $1111$, only the rightmost symbol can overlap an occurrence of $w_n$.
  Thus, the three leftmost symbols overlap the non-$w_n$ part of a $n$-interval of some length $k$.
  This interval cannot be common, as no word of $B_k$ contains more than two consecutive $1$-symbols.
  Thus, $(F_n \mu_p)([1111])$ is at most the probability that the $n$-interval that contains the origin is not common.

  An $n$-interval is not common if it is short, long, or medium but not in $A_k$.
  \begin{itemize}
  \item The probability of being short is at most the probability of seeing an occurrence of $w_n$ in a $\mu_p$-random word of length $2n/p$, which is upper-bounded by $2n/p \cdot \mu_p(w_n) = 2n p^{n+1} (1-p)^{n+1}$.
  \item Since the mean length of an $n$-interval is $1/\mu_p(w_n)$, by Markov's inequality the probability of being long is at most $1/n$.
  \item By Hoeffding's bound~\cite{Ho63}, the probability of a medium $n$-interval to not be common is at most $2 \exp(-pn)$.
  \end{itemize}
  Each of these tends to $0$ as $n \longrightarrow \infty$, hence $(F_n \mu_p)([1111])$ does as well.
  \qed
\end{proof}

Orthogonally to Theorem~\ref{thm:NotFullSupport}, we can also weaken Question~\ref{q:FullSupportLimits} by removing the condition that $\mu$ be a product measure.
We show that the answer becomes negative if we only assume ergodicity.

\begin{theorem}
\label{thm:NotFullSupportXOR}
Let $F$ be the two-neighbor XOR automaton.
There exists an ergodic full-support measure $\mu$ on $\{0,1\}^{\Z}$ such that $(F^n \mu)_{n \in \N}$ weak-* converges to the Dirac measure $\delta_0$.
\end{theorem}

Note that without the ergodicity requirement, the result would be much easier.
For example, we could choose $\mu$ as the measure that picks $n$ according to some full-support distribution on $\N$, and then a uniformly random spatially $2^n$-periodic configuration.
Because $F^{2^n}(x) = {}^\infty 0^\infty$ when $x$ is spatially $2^n$-periodic, $F^n \mu$ indeed weak-* converges to $\delta_0$.

\begin{proof}
The idea is that we pick an odometer structure on $\Z$ by grouping it into intervals of length $2$, then grouping those into intervals of length $8$, and so on.
The lengths form the sequence $2^{t(n)}$ where $t(n) = n(n+1)/2$.
For each such interval we flip a biased coin, and with high probability (tending to $1$ as $n$ grows) constrain it to be locally $2^{t(n-1)}$-periodic.
After $2^{t(n-1)}$ applications of XOR, the configuration is mostly $0$s.

We define a measure $\nu$ on an auxiliary dynamical system $(X, T)$ and a measurable factor map $\phi : (X,T) \to (\{0,1\}^\Z, \sigma)$, and let $\mu = \phi \nu$.
The space $X$ is a subset of an infinite product of subshifts: $X \subset \{0,1\}^\Z \times \prod_{n \geq 1} X_n$, where each $X_n$ is a subshift over the alphabet $\Sigma_n = \{0, 1, \ldots, 2^{t(n)}-1\}$ defined as the orbit of the periodic point ${}^\infty 0 1 2 \cdots (2^{t(n)}-1)^\infty$.
Points of $X$ are denoted $x = (x^*, x^1, x^2, \ldots)$ where $x^* \in \{0,1\}^\Z$ and $x^n \in X_n$.
We require for all $n \in \N$ and $i \in \Z$ that $x^n_i \equiv x^{n+1}_i \bmod 2^{t(n)}$.
The dynamics $T$ is just the left shift on each component, and the factor map $\phi$ is the projection $\phi(x) = x^*$.
% TODO: draw figure?

Let $(p_n)_{n \geq 1}$ be an increasing sequence with $0 < p_n < 1$, $p_n \longrightarrow 1$ and $\prod_n p_n = 0$, such as $p_n = 1 - 1/(n+1)$.
We define the measure $\nu$ on $X$ by describing how to sample a $\nu$-random configuration $x \in X$.
We construct the components $x^n$ sequentially.
When $x^1, \ldots, x^{n-1}$ have been fixed, there are $2^n$ choices for the configuration $x^n$.
We pick one of them with uniform probability.

For $n \geq 1$, each position $i \in \Z$ with $x^n_i = 0$ defines an \emph{$n$-block}, which is the interval $\{i, i+1, \ldots, i+2^{t(n)}-1\}$.
Each singleton $\{i\}$ is a $0$-block.
For $n \geq 1$, each level-$n$ block consists of $2^n$ consecutive $(n-1)$-blocks, called its \emph{children}.

To sample $x^*$, we choose some $0$-block $i \in \Z$ as the \emph{anchor} and pick $x^*_i$ to be $0$ or $1$ each with probability $1/2$.
Suppose then that we have picked the value of $x^*|_I$ for some $n$-block $I \subset \Z$.
Let $J$ be the $(n+1)$-block containing $I$, and let $K_1, \ldots, K_{2^n-1}$ be its children other than $I$.
Then, with probability $p_n$ we choose $x^*|_{K_i} = x^*|_I$ for all $i$, and with probability $1-p_n$ we pick each $x^*|_{K_i}$ by choosing the leftmost $0$-block of $K_i$ as a temporary anchor and repeating the above process with independent random choices.
With probability $1$, the entire configuration $x^*$ is eventually specified.
This concludes the definition of $\nu$.

We prove that the distribution of $\nu$ is independent of the choice of the anchor.
Let $i_1 \leq i_2$ be two anchors.
Almost surely the odometer components are such that $i_1$ and $i_2$ are contained in some $n$-block $I$.
We proceed by induction on the minimal such $n$.
For $n = 0$ we have $i_1 = i_2$ and the claim is trivial.
For $n \geq 1$, let $I_1 \neq I_2$ be the children of $I$ that contain $i_1$ and $i_2$, respectively.
The process that starts at $i_1$ samples $I_1$, which by the induction hypothesis has the same distribution as $I_2$ in the $i_2$-originated process.
With probability $p_n$, the $i_1$-process copies the contents of $I_1$ to every child of $I$, and with probability $1-p_n$ it samples the other children independently, and by the induction hypothesis, identically to the $i_2$-process.
The $i_2$-process behaves symmetrically, and we see that they produce identical distributions on $x^*|_I$.
After $x^*|_I$ has been sampled, both processes are identical, so the distributions over $x^*$ are identical.
It also follows that $\nu$ is $T$-invariant, and its image $\mu$ is shift-invariant.

The measure $\mu$ has full support: for any word $w \in \{0,1\}^{2^{t(n)}}$, there is a nonzero probability for $\{0, \ldots, 2^{t(n)}-1\}$ to be an $n$-block, none of its sub-blocks to be copied from their neighbors, and the randomly chosen bits to result in $w$.

We now prove that $F^n \mu \longrightarrow \delta_0$.
Recall that the two-neighbor XOR has the property that $F^{2^k}(x)_i \equiv x_i + x_{i+2^k} \bmod 2$ for all $k \geq 0$, $x \in \{0,1\}^\Z$ and $i \in \Z$.
Consider a $\nu$-random configuration $x = (x^*, x^1, x^2, \ldots)$ and let $n \geq 1$.
Let $I$ be the $(n+2)$-block containing the position $0$.
With probability at least $1 - 2^{-(n+1)} - 2^{-n(n+1)}$, the $n$-block containing $0$ is not among the $2^{n+1} + 2^n$ the rightmost $n$-blocks in $I$.
Independently, with probability at least $p_n p_{n+1}$, all $n$-blocks in $I$ have equal contents in $x^*$.
If both events happen, we have $x^*_i = x^*_{2^{t(n)}+i}$ for each $0 \leq i < 2^{t(n+1)}$, and hence $F^k(x^*) = 0$ for all $2^{t(n)} \leq k < 2^{t(n+1)}$.
Both probabilities tend to $1$ as $n$ grows, so $F^n \mu$ weak-* converges to $\delta_0$.

We show that $\nu$ is ergodic, which means that $\frac{1}{N} \sum_{k=0}^{N-1} \nu(A \cap T^{-k}(B))$ converges to $\nu(A) \nu(B)$ as $N$ grows, for all Borel sets $A, B \subseteq X$.
By a standard argument (based on the fact that for every Borel set $A$ and $\epsilon > 0$ there exists a finite union of cylinder sets $C$ with $\nu(A \mathop{\triangle} C) < \epsilon$) it suffices to prove this when $A$ and $B$ are cylinder sets with domain $D = \{0, 1, \ldots, 2^{t(n)}-1\}$ on the components $x^*$ and $x^1, \ldots, x^n$ for some $n \geq 0$.
For $0 \leq i < 2^{t(n)}$, let $P_i = \{ x \in X \mid x^n_0 = i \}$.
Since $A$ and $B$ are cylinder sets, we have $A \subseteq P_i$ and $B \subseteq P_j$ for some $i, j$.
We assume $i = j = 0$, the general case being similar.
Then we have $\nu(A \cap T^{-k}(B) \mid P_i) = 0$ whenever $i \neq 0$ or $k \not\equiv 0 \bmod 2^{t(n)}$, and $\nu(A \mid P_0) = 2^{t(n)} \nu(A)$ and analogously for $B$.

Let us condition on the event $P_0$.
Fix a large $M \geq n$ and let $0 \leq \ell < 2^{t(M) - t(n)}$ vary.
The domains of $A$ and $T^{-\ell 2^{t(n)}}(B)$ coincide with some $n$-blocks, and lie in some $M$-blocks (possibly the same one).
The sequence of directions (which child to choose) to reach the domain of $A$ from its $M$-block is given by the representation of $\ell_A$ in the mixed base $b = (2^{M-1}, 2^{M-2}, \ldots, 2^{n+1})$, where $\ell_A 2^{t(n)}$ is the distance from the left end of the $M$-block to the left end of $A$.
We define $\ell_B$ similarly for the domain of $T^{-\ell 2^n}(B)$.
Then $\ell_B - \ell_A \equiv \ell \mod 2^{t(M) - t(n)}$, so when $M$ is large compared to $n$, for a large fraction of $0 \leq \ell < 2^{t(M) - t(n)}$ the base-$b$ representations of $\ell_A$ and $\ell_B$ differ in at least $(M - n)/3$ digits; call this set $I \subset \{0, \ldots, M-n\}$.
With probability at least $1 - p_{n} p_{n+1} \cdots p_{n+(M-n)/3}$, which tends to $1$ as $M$ grows, for some $i \in I$, the children of the $(n+i)$-block containing $A$ are sampled independently.
Thus, regardless of whether the $(M+n)$-blocks that contain $A$ and $T^{-\ell 2^{t(n)}}(B)$ are the same, copies of each other, or independently sampled, for such an $\ell$ the $n$-blocks that make up the domains of $A$ and $T^{-\ell 2^n}(B)$ are sampled independently.
This implies $\nu(A \cap T^{-\ell 2^{t(n)}}(B) \mid P_0) = \nu(A \mid P_0) \nu(T^{-\ell 2^{t(n)}}(B) \mid P_0) = 2^{2 t(n)} \nu(A) \nu(B)$.

Altogether this implies that
\begin{align*}
{} & \frac{1}{2^{t(M)}} \sum_{k=0}^{2^{t(M)}-1} \nu(A \cap T^{-k}(B)) \\
{} = {} & \frac{1}{2^{t(M)}} \sum_{k=0}^{2^{t(M)}-1} \sum_{i=0}^{2^{t(n)}-1} \nu(A \cap T^{-k}(B) | P_i) \nu(P_i) \\
{} = {} & \frac{1}{2^{t(M)+t(n)}} \sum_{\ell=0}^{2^{t(M) - t(n)}-1} \nu(A \cap T^{-\ell 2^{t(n)}}(B) | P_0)
\end{align*}
is close to $\nu(A) \nu(B)$.
Hence $\nu$ is $T$-ergodic, and $\mu$ is shift-ergodic.
\qed
\end{proof}

Again somewhat orthogonally to both of the above theorems, we can produce a different limit behavior by iterating XOR on a single initial ergodic measure, at the cost of having to pass to a fixed subsequence.

\begin{theorem}
\label{thm:XORall1}
Let $F$ be the two-neighbor XOR automaton and $0 \leq \alpha \leq 1$.
There exists an ergodic full-support measure $\mu$ on $\{0,1\}^\Z$ such that the sequence $(F^{2^{n(n+1)/2}} \mu)_{n \in \N}$ weak-* converges to $\alpha \delta_0 + (1 - \alpha)\delta_1$.
\end{theorem}

\begin{proof}
We use the same construction of the measures $\nu$ and $\mu$ as in the proof of Theorem~\ref{thm:NotFullSupportXOR}, but with the following modification.
When sampling $x^*$, suppose we have sampled an $n$-block $I$ containing the anchor, so that the contents of $x^*|_I$ are fixed, and are considering the $(n+1)$-block $J$.
With probability $1-p_n$, we sample each child of $J$ independently, as before.
With probability $\alpha p_n$, all children of $J$ are filled with $x^*|_I$.
With probability $(1-\alpha)p_n$, the children of $J$ alternate between $x^*|_I$ and its symbolwise negation.

This distribution is also independent of the position of the anchor; the proof is the same as in Theorem~\ref{thm:NotFullSupportXOR}, with the additional observation that the distribution of the contents of $n$-blocks in $x^*$ is invariant under cellwise negation.
This latter property can be easily proved by induction on $n$.
The proofs of the shift-invariance, ergodicity and full support of $\mu$ are then exactly as before.

As for the limit points, take a $\nu$-random configuration $x = (x^*, x^1, x^2, \ldots)$ and let $n \geq 1$.
Let $I$ be the $(n+1)$-block containing the origin.
With probability $1 - 2^{-n+1}$, the $n$-block containing the origin is not one of the two rightmost children of $I$.
Independently, with probability $\alpha p_n$, the children of $I$ have equal contents in $x^*$, and with probability $(1-\alpha)p_n$ they alternate between some word and its symbolwise negation.
As $n$ grows, the probability of $x^*_i = x^*_{2^{t(n)}+i}$, and hence $F^{2^{t(n)}}(x)_i = 0$, for all $0 \leq i < 2^{t(n)}$ tends to $\alpha$, and the probability of $F^{2^{t(n)}}(x)_i = 1$ for all $0 \leq i < 2^{t(n)}$ tends to $1-\alpha$.
This implies the claim.
\qed
\end{proof}

\section{Correlation} % Maybe rename?
%We could merge sections 3 and 4 given the new direction of the paper
In this section, we introduce a notion called correlation and prove that, among surjective CA of a given radius $r$, the identity CA has the highest $(A,A)$-correlation of order 1 and of any high enough order.

Let us first provide some motivation. Suppose that you are trying to prove Equation~(\ref{eq:faruniform}) in the particular case where $\mu$ is a product measure. Without loss of generality, we consider $F\mu([b])$ where $b\in \Sigma$ is a single letter.

\begin{equation}\label{eq:correlationMotivation}F\mu([b]) = \mu(F^{-1}([b]) = \sum_{f(w) = b} \mu[w] = \sum_{f(w) = b} \prod_{a\in\Sigma}(\mu[a])^{|w|_A}\end{equation}

Thus the values $F\mu([b])$ depend of the number and repartition of occurrences of each letter $a$ in the pre-images of $b$ under $F$, that is, the values of $|w|_A$ for $w\in f^{-1}(b)$ (notice that, by the balance property, the number of such words does not depend on $b$), and we would be looking to compare this repartition with the case $F=\ID$. As a first step, we try to count the total number of occurrences in all preimages without considering the repartition: this is the intuition behind correlation, which is essentially an additive version of Equation~(\ref{eq:correlationMotivation}).

\subsection{Definition and basic properties}

\begin{definition}
\label{def:correlation}
  Let $F$ be a CA of radius at most $r$, and let $A, B \subseteq \Sigma$.
  For $0 \leq k \leq r+1$, we denote $N^A_B(F,r,k) = |\{ w \in \Sigma^{r+1} \mid F(w)_0 \in B, |w|_A = k \}|$.
  The \emph{$(A,B,r)$-histogram} of $F$ is the list $H^A_B(F,r) = (N^A_B(F,r,k))_{0 \leq k \leq r+1}$.
  
  The \emph{$(A,B,r)$-correlation of order $m$} of $F$ is the number \[C^A_B(F,r,m) = \sum_{k=0}^{r+1} k^m N^A_B(F,r,k).\]
  The $(A,B,r)$-correlation of order $1$ is simply called the $(A,B,r)$-correlation.
  
  We may drop $r$ from all of these definitions (for example, we may discuss the $(A,B)$-correlation and denote it by $C^A_B(F,1)$) if it is clear from the context.
\end{definition}

In other words, $N^A_B(F,r,k)$ counts the number of words that contain exactly $k$ symbols from $A$ and map to a symbol of $B$ under $f$.
The $(A,B,r)$-histogram lists these counts for each value $k \in \{0, \ldots, r+1\}$.
The $(A,B,r)$-correlation of order $m$, or $C^A_B(F,r,m)$, is the $m$'th moment of the $(A,B,r)$-histogram.

We also have the simpler formula
\begin{equation}
  \label{eq:correlation}
  C^A_B(F,r,m) = \sum_{w \in f^{-1}(B)} |w|_A^m
\end{equation}
where $f : \Sigma^{r+1} \to \Sigma$ is the radius-$r$ local function of $F$.
In particular, the $(A,B,r)$-correlation is the total number of $A$-symbols in all preimages of $B$-symbols, and (with the convention that $0^0 = 1$) the $(A,B,r)$-correlation of order $0$ is $C^A_B(F, r, 0) = |f^{-1}(B)|$.
If $F$ is a surjective CA, then its local function is balanced, so $C^A_B(F,r,0) = |B| \cdot |\Sigma|^r$.

\begin{example}
\label{ex:correlation1}
  Consider the radius-2 CA $F$ with state set $\Sigma = \{0,1,2\}$ defined by the local rule
  \[
    f(a,b,c) =
    \begin{cases}
      2, & \text{if~} b = c, \\
      1, & \text{if~} b \neq c \text{~and~} a = 0, \\
      0, & \text{otherwise.}
    \end{cases}
  \]
  For $A = \{0\}$ and $B = \{0,2\}$, the quantity $N^A_B(F,1)$ is the number of words $w \in \Sigma^3$ that satisfy $f(w) \in \{0,2\}$ and $|w|_0 = 1$.
  These words are exactly $011$, $022$, $101$, $102$, $110$, $120$, $201$, $202$, $210$, and $220$.
  Hence $N^A_B(F,1) = 10$.

  The full $(A,B)$-histogram of $F$ is $(8, 10, 2, 1)$.
  Its sum, $8 + 10 + 2 + 1 = 21$, equals the number of words in $f^{-1}(B)$.
  The first moment of the histogram, $0 \cdot 8 + 1 \cdot 10 + 2 \cdot 2 + 3 \cdot 1 = 17$, equals $C^A_B(F,1)$, the $(A,B)$-correlation of $F$.
\end{example}

The $(A,B)$-correlation of a CA is defined for a fixed radius.
We also present a normalized definition that is independent of the radius.

%\begin{lemma}
%    \label{lem:correlationRadius}
%    Let $F$ be a radius-$r$ CA on $\Sigma^\Z$, and let $A, B \subseteq \Sigma$ and $k \geq 0$.
%    Then
%    \begin{equation}
%        \label{eq:correlationRadius}
%        C^A_B(F, r+k+1, 1) = |\Sigma| \cdot C^A_B(F, r+k, 1) + |\Sigma|^k \cdot |A| \cdot |f^{-1}(B)|
%    \end{equation}
%    where $f : \Sigma^{r+1} \to \Sigma$ is the local rule of $F$.
%\end{lemma}

%\begin{proof}
%    Suppose first $k = 0$.
%    From Equation~\eqref{eq:correlation} we have
%    \begin{align*}
%        C^A_B(F, r+1, 1) = {} & \sum_{w \in f^{-1}(B)} \sum_{s \in \Sigma} |w s|_A \\
%        {} = {} & \sum_{w \in f^{-1}(B)} |\Sigma| \cdot |w|_A + |A| \\
%        {} = {} & |\Sigma| \cdot C^A_B(F, r, 1) + |A| \cdot |f^{-1}(B)|,
%    \end{align*}
%    as claimed.
%    For $k \geq 1$, the result follows from the $k = 0$ case and the fact that $|f_k^{-1}(B)| = |\Sigma|^k \cdot |f^{-1}(B)|$, where $f_k : \Sigma^{r+k+1} \to \Sigma$ is the radius-$(r+k)$ local rule of $F$.
%    \qed
%\end{proof}

\begin{lemma}
    \label{lem:correlationRadius-m}
    Let $F$ be a radius-$r$ CA on $\Sigma^\Z$, $A, B \subseteq \Sigma$ and $m, k \geq 0$.
    Then
    \begin{equation}
        \label{eq:correlationRadius-m}
        C^A_B(F, r+k+1, m) = |\Sigma| \cdot C^A_B(F, r+k, m) + |A| \sum_{i=0}^{m-1}\binom{m}{i} C^A_B(F, r+k, i).
    \end{equation}
\end{lemma}

\begin{proof}
Let $f_k : \Sigma^{r+k+1} \to \Sigma$ be the radius-$(r+k)$ local rule of $F$.
From Equation~\eqref{eq:correlation} we compute
\begin{align*}
  C^A_B(F, r+k+1, m) = {} & \sum_{w \in f_k^{-1}(B)} \sum_{s \in \Sigma} |w s|_A^m \\
        {} = {} & \sum_{w \in f_k^{-1}(B)} |A|\cdot(1 + |w|_A)^m + |\Sigma\setminus A|\cdot |w|_A^m\\
        {} = {} & |\Sigma| \cdot C^A_B(F, r+k, m) + |A| \sum_{w \in f_k^{-1}(B)}\sum_{i=0}^{m-1} \binom{m}{i}|w|_A^i \\
        {} = {} & |\Sigma| \cdot C^A_B(F, r+k, m) + |A| \sum_{i=0}^{m-1}\binom{m}{i} C^A_B(F, r+k, i),
\end{align*}
which is exactly the claim.
\qed
\end{proof}

\begin{definition}
    Let $F$ be a CA on $\Sigma^\Z$ and $A, B \subseteq \Sigma$.
    We extend the definition of $C^A_B(F, r, m)$, the $(A,B,r)$-correlation of order $m$ of $F$, to all $m, r \geq 0$ as follows.
    For $r$ at least the radius of $F$, we use Definition~\ref{def:correlation}.
    For smaller $r$, we apply Equation~\ref{eq:correlationRadius-m} recursively, with $k$ negative.
    The \emph{normalized $(A,B)$-correlation} of order $m$ of $F$ is $\bar{C}^A_B(F,m) = C^A_B(F, 0, m)$.
    We denote $\bar{C}^A_B(F,1) = \bar{C}^A_B(F)$.
\end{definition}

By induction on $0 \leq k \leq r$ and using Equation~\ref{eq:correlationRadius-m}, we can prove that
\[
C^A_B(F, r-k, m) = \frac{C^A_B(F, r, m)}{|\Sigma|^k} - |A| \sum_{i=0}^{m-1} \binom{m}{i} \sum_{t=1}^k \frac{C^A_B(F,r-t,i)}{|\Sigma|^{k+1-t}}.
\]
With $m = 1$ and $k = r$, since $C^A_B(F, r-t, 0) = |f^{-1}(B)| / |\Sigma|^t$, this gives
\begin{equation}
    \label{eq:normalizedCorrelation}
    \bar{C}^A_B(F) = \frac{C^A_B(F, r, 1)}{|\Sigma|^r} - r \frac{|A| \cdot |f^{-1}(B)|}{|\Sigma|^{r+1}}.
\end{equation}
We could produce analogous formulas for $\bar{C}^A_B(F, m)$ with $m \geq 2$, expressed in terms of $C^A_B(F,r,m)$ and $\bar{C}^A_B(F,i)$ for $0 \leq i < m$, but they quickly become unwieldy.
Note that $\bar{C}^A_B(F)$ may be negative.

If $F$ is a radius-$r$ surjective CA, the balance property implies $|f^{-1}(B)| = |B| \cdot |\Sigma|^r$.
Then Equation~\ref{eq:normalizedCorrelation} becomes
\begin{equation}
\label{eq:normalizedCorrelationSurj}
\bar{C}^A_B(F) = \frac{C^A_B(F, r, 1)}{|\Sigma|^r} - r \frac{|A| \cdot |B|}{|\Sigma|}.
\end{equation}
In particular, for a fixed $r$ and surjective $F$ the conversion from $C^A_B(F, r, 1)$ to $\bar{C}^A_B(F)$ is given by a monotone function.

\begin{example}
    Let us continue Example~\ref{ex:correlation1}, where we computed $C^A_B(F,2,1) = 17$ and $|f^{-1}(B)| = 21$.
    Plugging into Equation~\ref{eq:normalizedCorrelation} these values and $|\Sigma| = 3$, $|A| = 1$ and $r = 2$, we get $\bar{C}^A_B(F) = 3^{-2} \cdot 17 - 2 \cdot 3^{-3} \cdot 1 \cdot 21 = 1/3$.
\end{example}

\begin{example}
    For non-surjective CA of radius $r \geq 1$, the conversion from $(A,B,r)$-correlation to normalized $(A,B)$-correlation may not be monotone.
    Let $\Sigma = \{0,1\}$ and $r = 2$, and let $F$ and $G$ be defined by the local functions
    \[
        f(w) =
        \begin{cases}
            1 & \text{if~} |w|_1 \leq 2 , \\
            0 & \text{otherwise}
        \end{cases}
        \qquad
        g(w) =
        \begin{cases}
            1 & \text{if~} |w|_1 = 2, \\
            0 & \text{otherwise}
        \end{cases}
    \]
    where $w \in \{0,1\}^3$.
    Now we have $C^1_1(F, 2, 1) = 9 > 6 = C^1_1(G, 2, 1)$ and $\bar{C}^1_1(F) = 1/2 < 3/4 = \bar{C}^1_1(G)$.
\end{example}

The following result can be seen as evidence that our definition of the normalized $(A,B)$-correlation is the ``correct'' way to remove the dependence on $r$ from the $(A,B,r)$-correlation.

\begin{proposition}
\label{prop:randomCA}
    Let $A, B \subseteq \Sigma$.
    For each $r \geq 0$, the average of the normalized $(A,B)$-correlation of order $1$ among all radius-$r$ CA on $\Sigma^\Z$ is $|A| \cdot |B| / |\Sigma|$.
\end{proposition}

\begin{proof}
    %To use Equation~\ref{eq:normalizedCorrelation}, we compute the average of $C^A_B(F, r, 1)$ and $|f^{-1}(B)|$ among radius-$r$ cellular automata $F$ and their local rules $f$.
    For a radius-$r$ CA $F$ with local rule $f$, let $I_f$ be the indicator function of $f^{-1}(B)$, that is, $I_f(w) = 1$ if $w \in f^{-1}(B)$ and $I_f(w) = 0$ otherwise for all $w \in \Sigma^{r+1}$.
    Then $|f^{-1}(B)| = \sum_{w \in \Sigma^{r+1}} I_f(w)$.
    From Equation~\ref{eq:correlation} we have
    \[
    C^A_B(F,r,1) = \sum_{w \in f^{-1}(B)} |w|_A = \sum_{w \in \Sigma^{r+1}} I_f(w) \cdot |w|_A.
    \]
    Combining these with Equation~\ref{eq:normalizedCorrelation}, we have
    \[
    \bar{C}^A_B(F) = \sum_{w \in \Sigma^{r+1}} I_f(w) \left( |\Sigma|^{-r} |w|_A - r |\Sigma|^{-r-1} |A| \right).
    \]
    For each $w$, the average of $I_f(w)$ over all $f$ is $|B|/|\Sigma|$.
    The average of $\bar{C}^A_B(F)$ over all $f$ is thus
    \[
    \sum_{w \in \Sigma^{r+1}} \frac{|B|}{|\Sigma|}\left( |\Sigma|^{-r} |w|_A - r |\Sigma|^{-r-1} |A| \right) = |B|(r+1)\frac{|A|}{|\Sigma|} - r \frac{|A| \cdot |B|}{|\Sigma|} = \frac{|A| \cdot |B|}{|\Sigma|},
    \]
    as claimed.
    \qed
\end{proof}

\subsection{Upper Bound for Order 1}

In this section we show that the identity CA has the highest normalized $(A,A)$-correlation among all surjective CA. A similar computation would yield that the minimum is reached by any symbol permutation $\tau$ that minimizes $|\tau(A) \cap A|$. When $A\cap B=\emptyset$ and $|A|=|B|$, the identity has the lowest $(A,B)$-correlation, while the highest is reached by any symbol permutation $\tau$ with $\tau(A) = B$.

A direct computation gives $\bar{C}^A_A(\ID) = |A|$.
As expected, it is higher than $|A|^2 / |\Sigma|$, the average of the normalized $(A,A)$-correlation of a randomly chosen CA as per Proposition~\ref{prop:randomCA}.
Equation~\ref{eq:normalizedCorrelationSurj} now gives for all $r \geq 0$ that
\begin{equation}
\label{eq:id}
C^A_A(\ID,r,1) = (|\Sigma| + r |A|) |A| \cdot |\Sigma|^{r-1}.
\end{equation}

\begin{lemma}
\label{lem:WeightSquareSum}
  For each integer $n \geq 0$ and real $a \geq 0$, we have
  \[
    \sum_{k=0}^n k^2 a^k \binom{n}{k} = n a (n a + 1) (a + 1)^{n-2}.
  \]
\end{lemma}

\begin{proof}
  We apply double counting.
  The left hand side counts the ways of choosing a subset $S \subset \{1, \ldots, n\}$ and then a function $f : \{0,1\} \to S$, where each choice is given the weight $a^{|S|}$.
  
  We can also choose first the function $f : \{0,1\} \to \{1, \ldots, n\}$ and then the remaining elements of $S$.
  There are $n (n-1)$ choices for an injective $f$, and each contributes $\sum_{m=0}^{n-2} a^{n-m} \binom{n-2}{m} = a^2 (a+1)^{n-2}$ to the total weight.
  For a non-injective $f$, there are $n$ choices, and each contributes $\sum_{m=0}^{n-1} a^{n-m} \binom{n-1}{m} = a (a+1)^{n-1}$ to the total weight.
  Thus the total weight is
  \[
    n(n-1) \cdot a^2 (a+1)^{n-2} + n \cdot a (a+1)^{n-1} = n a (n a + 1) (a + 1)^{n-2},
  \]
  which is exactly the right hand side.
  \qed
\end{proof}

\begin{theorem}
  \label{thm:OneDomination}
  If $F$ is a surjective CA of radius at most $r$, then $C^A_A(F,r,1) \leq C^A_A(\ID,r,1)$, and hence $\bar{C}^A_A(F) \leq \bar{C}^A_A(\ID) = |A|$.
\end{theorem}

\begin{proof}
  For $n \geq 1$, denote by $F_n : \Sigma^{n+r} \to \Sigma^n$ the CA $F$ applied to words of length $n+r$.
  We define the $(A,A,r)$-correlation of $F_n$ as $C^A_A(F_n,r,1) = \sum_{w \in \Sigma^{n+r}} |w|_A \cdot |F_n(w)|_A$.
  We prove the desired bound by expressing $C^A_A(F_n,r,1)$ in terms of the correlation $C^A_A(F,r,1)$, proving an upper bound on $C^A_A(F_n,r,1)$ that depends on $n$, and letting $n$ go to infinity.
  
  For the first step, we define the set of words $W^i_n = \{ w \in \Sigma^{n+r} \mid F_n(w)_i \in A \}$.
  Since $r$ is the radius of $F$, a given word $w \in \Sigma^{n+r}$ belongs to $W^i_n$ if and only if $F(w_{[i,i+r]}) \in A$; the other coordinates of $w$ are irrelevant.
  We rewrite the definition of $C^A_A(F_n,r,1)$ as
  \begin{align*}
    C^A_A(F_n,r,1)
    & {} = \sum_{i=0}^{n-1} \sum_{w \in W^i_n} |w|_A \\
    & {} = \sum_{i=0}^{n-1} \sum_{w \in W^i_n} ( |w_{[0,i-1]}|_A + |w_{[i,i+r]}|_A + |w_{[i+r+1,n+r-1]}|_A ) \\
    & {} = \sum_{i=0}^{n-1} \left( |\Sigma|^{n-1} C^A_A(F,r,1) + |A| |\Sigma|^r |\Sigma|^{n-1} (n-1) \frac{|A|}{|\Sigma|} \right) \\
    & {} = n |\Sigma|^{n-1} (C^A_A(F,r,1) + (n-1) |A|^2 |\Sigma|^{r-1}).
  \end{align*}
  The second-to-last equality follows from two considerations.
  The first term of the inner sum comes from $|w_{[i,i+r]}|_A$ for $w \in W^i_n$.
  Each subword $w_{[i,i+r]}$ occurs exactly $|\Sigma|^{n-1}$ times in the sum, since we have $n-1$ ``irrelevant'' coordinates that do not affect the condition $w \in W^i_n$.
  The second term comes from $|w_{[0,i-1]}|_A + |w_{[i+r+1,n+r-1]}|_A$ for $w \in W^i_n$, or in other words, the irrelevant coordinates.
  Each of the $|\Sigma|^{n-1}$ words over these coordinates occurs exactly $|F^{-1}(A)| = |A| |\Sigma|^r$ times, and their average $A$-weight is $(n-1) \frac{|A|}{|\Sigma|}$.
  With some rearraging, we obtain
  \begin{equation}
    \label{eq:LocalGlobal}
    C^A_A(F,r,1) = \frac{C^A_A(F_n,r,1)}{n |\Sigma|^{n-1}} - (n-1)|A|^2 |\Sigma|^{r-1}.
  \end{equation}
  
  Next, we prove the upper bound on $C^A_A(F_n,r,1)$.
  Define two integer vectors $\vec u_n = (|w|_A)_{w \in \Sigma^{n+r}}$ and $\vec v_n = (|F_n(w)|_A)_{w \in \Sigma^{n+r}}$ in $\R^{|\Sigma|^{n+r}}$, with the words in $\Sigma^{n+r}$ enumerated in some consistent order.
  It is immediate from the definition that $C^A_A(F_n,r,1) = \vec u_n \cdot \vec v_n$.
  Since $F_n$ is a balanced function, the weight of each word $w \in \Sigma^n$ is enumerated in $\vec v_n$ exactly $|\Sigma|^r$ times.
  Using these facts and the Cauchy-Schwarz inequality, we compute
  \begin{equation}
    \label{eq:CS}
    C^A_A(F_n,r,1) \leq \| u_n \| \cdot \| v_n \| = \sqrt{\left( \sum_{w \in \Sigma^{n+r}} |w|_A^2 \right) \left( |\Sigma|^r \sum_{v \in \Sigma^n} |v|_A^2 \right)}.
  \end{equation}
  Consider the rightmost sum $s = \sum_{v \in \Sigma^n} |v|_A^2$.
  Grouping the words by their $A$-weight and denoting $B = \Sigma \setminus A$, we have
  \begin{align*}
  s
  & {} = \sum_{k=0}^n k^2 |A|^k |B|^{n-k} \binom{n}{k}
       = |B|^n \sum_{k=0}^n k^2 \left(\frac{|A|}{|B|}\right)^k \binom{n}{k} \\
  & {} = n |B|^n \frac{|A|}{|B|} \left(n \frac{|A|}{|B|} + 1 \right) \left(\frac{|A|}{|B|} + 1 \right)^{n-2}
       = n |A| ((n - 1) |A| + |\Sigma|) |\Sigma|^{n-2}
  \end{align*}
  by Lemma~\ref{lem:WeightSquareSum}.
  Applying this to both sums in~\eqref{eq:CS}, we obtain
  \begin{equation}
    \label{eq:CS2}
    C^A_A(F_n,r,1) \leq |A| |\Sigma|^{n+r-2} \sqrt{n(n+r)((n-1)|A| + |\Sigma|)((n+r-1)|A| + |\Sigma|)}.
  \end{equation}
  
  We now combine~\eqref{eq:LocalGlobal} and~\eqref{eq:CS2} to get an upper bound for $C^A_A(F,r,1)$, which holds for all $n \geq 1$.
  For the sake of notational convenience, we express the bound in terms of the quantity $C^A_A(F,r,1)/(|A| |\Sigma|^{r-1}) + |A|$:
  \begin{align*}
    \frac{C^A_A(F,r,1)}{|A| |\Sigma|^{r-1}} + |A|
    & {} \leq \sqrt{\frac{(n+r)((n-1)|A| + |\Sigma|)((n+r-1)|A| + |\Sigma|)}{n}} - n |A| \\
    & {} = \frac{(n+r)((n-1)|A| + |\Sigma|)((n+r-1)|A| + |\Sigma|) - |A|^2 n^3}{\sqrt{n(n+r)((n-1)|A| + |\Sigma|)((n+r-1)|A| + |\Sigma|)} + |A| n^2} \\
    & {} = \frac{2|A| ((r-1)|A| + |\Sigma|) n^2 + O(n)}{|A|n^2 (\sqrt{o(n) + 1} + 1)} \stackrel{n \to \infty}{\longrightarrow} (r-1)|A| + |\Sigma|
  \end{align*}
  The claim follows from this bound and Equation~\ref{eq:id}.
  \qed
\end{proof}

\begin{example}
  Theorem~\ref{thm:OneDomination} generally does not hold for non-surjective CA.
  Namely, let $A \subsetneq \Sigma$, and let $F$ be a radius-$r$ local rule whose image is contained in $A$.
  Its $(A,A,r)$-correlation is $C^A_A(F,r,1) = \sum_{w \in \Sigma^{r+1}} |w|_A = (r+1) |A| \cdot |\Sigma|^r$, which is strictly larger than $C^A_A(\ID,r, 1) = (|\Sigma| + r|A|)|A| \cdot |\Sigma|^{r-1}$.
\end{example}

\subsection{Upper Bound for High Enough Orders}

\begin{lemma}
\label{lem:PreimMeasure}
Let $F$ be a radius-$r$ CA on $\Sigma^\Z$, let $A \subseteq \Sigma$ and $0 < p < 1$, and let $\mu_{A,p}$ be the product measure with $\mu_{A,p}([a]) = p/|A|$ for all $a \in A$ and $\mu_{a,p}([b]) = (1-p)/(|\Sigma|-|A|)$ for all $b \notin A$.
Then we have
\begin{equation}
\label{eq:PreimMeasure}
  \sum_{a\in A}F\mu_{A,p}([a]) = \sum_{\ell = 0}^{r+1} p^{\ell} \left( \sum_{k=0}^{\ell} \frac{ (-1)^{\ell-k} \binom{r+1-k}{r+1-\ell} N^A_A(F,r,k) }{ |A|^k(|\Sigma|-|A|)^{r+1-k} } \right).
\end{equation}
\end{lemma}

In particular, this implies that the $(A,A)$-histogram of a CA $F$ completely determines the values $\sum_{a\in A}F\mu_{A,p}([a])$ for $0 < p < 1$.

\begin{proof}
We compute
  \begin{align*}
      \sum_{a\in A}F\mu_{A,p}([a]) = {} & \sum_{a\in A}\sum_{w \in F^{-1}([a])} \prod_{i=0}^r \mu_{A,p}(w_i) \\
      {} = {} & \sum_{k = 0}^{r+1} \sum_{a\in A}\sum_{\substack{w \in F^{-1}([a]) \\ |w|_A = k}} \left( \frac{p}{|A|} \right)^k \left( \frac{1 - p}{|\Sigma|-|A|} \right)^{r+1-k} \\
      {} = {} & \sum_{k = 0}^{r+1}\frac{N^A_A(F,k)}{|A|^k (|\Sigma|-|A|)^{r+1-k}} \cdot p^k (1 - p)^{r+1-k} \\
      {} = {} & \sum_{k = 0}^{r+1}\frac{N^A_A(F,k)}{|A|^k (|\Sigma|-|A|)^{r+1-k}} \cdot p^k \sum_{\ell=0}^{r+1-k} \binom{r+1-k}{\ell} (-p)^{r+1-k-\ell} \\
      {} = {} & \sum_{\ell = 0}^{r+1} p^{r+1-\ell} \left( \sum_{k=0}^{r+1-\ell} \frac{ (-1)^{r+1-k-\ell} \binom{r+1-k}{\ell} N^A_A(F,k) }{ |A|^k (|\Sigma|-|A|)^{r+1-k} } \right).
  \end{align*}
  The change of variables $\ell \to r+1-\ell$ gives the claim.
  \qed
\end{proof}

\begin{theorem}
  \label{thm:HighDomination}
  If $F$ is a surjective CA of radius at most $r$, then $C^A_A(F,r,m) \leq C^A_A(\ID,r,m)$ and $\bar{C}^A_A(F,m) \leq \bar{C}^A_A(\ID,m)$ for large enough $m$ (depending on $r$).
\end{theorem}

\begin{proof}
  If $F$ and $\ID$ have the same $(A,A,r)$-histogram, then they have the same $(A,A,r)$-correlation of every order $m$.
  If not, take the maximal $k_0 \in \{0, \ldots, r+1\}$ for which $N^A_A(F,r,k_0) \neq N^A_A(\ID,r,k_0)$.
  If $N^A_A(F,r,k_0) < N^A_A(\ID,r,k_0)$, then
  \begin{align*}
    C^A_A(\ID,r, m) - C^A_A(F, r,m) = {} & \sum_{i=0}^{r+1} (N^A_A(\ID,r,i) - N^A_A(F,r,i)) i^m \\
    {} = {} & (N^A_A(\ID,r,k_0) - N^A_A(F,r,k_0)) k^m + O((k_0-1)^m),
  \end{align*}
  which is strictly positive for large enough $m$.

  It remains to prove that $N^A_A(F,r,k_0) < N^A_A(\ID,r,k_0)$.
  Suppose for a contradiction that $N^A_A(F,r,k_0) > N^A_A(\ID,r,k_0)$.
  Choose a small $p > 0$ and consider the measure $\mu_{A,1-p}$ of Lemma~\ref{lem:PreimMeasure}.
  From the proof of the Lemma applied to $F$ and $\ID$,
  \begin{align*}
  	& \mu_{A,1-p}(F^{-1}(A)) - \mu_{A,1-p}(A) \\
  	{} = {} & \sum_{k = 0}^{r+1}\frac{N^A_A(F,r,k) - N^A_A(\ID,r,k)}{|A|^k (|\Sigma|-|A|)^{r+1-k}} \cdot (1-p)^k p^{r+1-k} \\
  	{} = {} & \frac{N^A_A(F,r,k_0) - N^A_A(\ID,r,k_0)}{|A|^{k_0} (|\Sigma|-|A|)^{r+1-k_0}} \cdot (1-p)^{k_0} p^{r+1-k_0} + O(p^{r+1-k_0+1}).
  \end{align*}
  When $p$ is small enough (in particular $p < |A|/|\Sigma|$), this is strictly positive, so $q = (F \mu_{A,1-p})(A) > \mu_{A,1-p}(A) = 1-p$.
  For the entropies, this implies
  \begin{align*}
    h(F \mu_{A,p}) \leq {} & - q \log\frac{q}{|A|} - (1-q) \log \frac{1-q}{|\Sigma| - |A|} \\
    {} < {} & - (1-p) \log \frac{1-p}{|A|} - p \log \frac{p}{|\Sigma| - |A|} \\
    {} = {} & h(\mu_{A,1-p}).
  \end{align*}
  But a surjective CA must preserve entropy, a contradiction.
  \qed
\end{proof}

\subsection{State Conservation}

Histograms also provide a characterization for the existence of certain conserved quantities for surjective CA.
Recall that a \emph{conserved quantity} of radius $r$ for a CA $F$ is a function $\phi : \Sigma^{r+1} \to \R$ such that for all $p \geq 1$ and all spatially $p$-periodic configurations $x \in \Sigma^\Z$, we have $\sum_{i=0}^{p-1} \phi(x_{[i, i+r]}) = \sum_{i=0}^{p-1} \phi(F(x)_{[i, i+r]})$.
Conserved quantities have a connection to invariant measures, which manifests as follows in the case of surjective CA and radius-$0$ conserved quantities.

\begin{theorem}[Theorem~1 in~\cite{KaTa12}]
  \label{thm:Siamak}
  A surjective CA $F : \Sigma^\Z \to \Sigma^\Z$ preserves a full-support product measure $\mu$ if and only if it conserves the radius-$0$ quantity $\phi(a) = - \log \mu(a)$.
\end{theorem}

\begin{theorem}
  \label{thm:Conserving}
  Let $F$ be a CA of radius at most $r$, and $A \subseteq \Sigma$.
  If $F$ conserves the total number of $A$-symbols, then $H^A_A(F,r) = H^A_A(\ID,r)$.
  The converse holds if $F$ is surjective.
\end{theorem}

\begin{proof}
  Suppose first that $F$ conserves the number of $A$-symbols.
  Choose a transcendental $0 < p < 1$ and let $\mu_{A,p}$ be the measure of Lemma~\ref{lem:PreimMeasure}.
  Then we have $\mu_{A,p}(F^{-1}(A)) = p$.
  Equation~\eqref{eq:PreimMeasure} now becomes
  \[
  	p = \mu_{A,p}(F^{-1}(A)) = \sum_{\ell = 0}^{r+1} p^{\ell} \left( \sum_{k=0}^{\ell} \frac{ (-1)^{\ell-k} \binom{r+1-k}{r+1-\ell} N^A_A(F,r,k) }{ |A|^k (|\Sigma|-|A|)^{r+1-k} } \right).
  \]
  This is a polynomial equation in $p$, and since $p$ is transcendental, the equation must be trivial: the coefficient of $p^{r+1-\ell}$ in the last sum must equal $1$ when $\ell = 1$, and $0$ otherwise.
  The case $\ell = 0$ fixes the value of $N^A_A(F,r,0)$, then $\ell = 1$ fixes the value of $N^A_A(F,r,1)$, and so on up to $N^A_A(F,r+1)$.
  Hence all $A$-preserving CA of radius at most $r$, including $\ID$, have the same $(A,A,r)$-histogram.

  Suppose then that $F$ is surjective and $H^A_A(F,r) = H^A_A(\ID,r)$, and let $0 < p < 1$ be arbitrary.
  From Lemma~\eqref{lem:PreimMeasure} we obtain $F \mu_{A,p}(A) = \ID \mu_{A,p}(A) = \mu_{A,p}(A) = p$.
  Since $F$ is surjective, it preserves entropy, so $h(F \mu_{A,p}) = h(\mu_{A,p})$.
  But $\mu_{A,p}$ is the unique measure of maximum entropy among all measures $\mu$ with $\mu(A) = p$.
  Hence $F$ preserves $\mu_{A,p}$.
  Theorem~\ref{thm:Siamak} implies that $F$ conserves the number of $A$-symbols.
  \qed
\end{proof}

\begin{example}
    The second claim of Theorem~\ref{thm:Conserving} does not hold for non-surjective cellular automata.
    Namely, consider the radius-$1$ CA $F$ on $\{0,1,2\}^\Z$ defined by the local rule
    \[
    f(a,b) =
    \begin{cases}
        0, & \text{if~} (a,b) = (1,0), \\
        1, & \text{if~} (a,b) = (0,1), \\
        a, & \text{otherwise.}
    \end{cases}
    \]
    We have $H^0_0(F,1) = (0,2,1) = H^0_0(\ID, 1)$, but $F$ does not conserve the total number of $0$-symbols, since $F({}^\infty (012)^\infty) = {}^\infty (112)^\infty$.
\end{example}

\section{Open Questions}

%We have proved that the $\mu$-limit set of a surjective cellular automata, 

We state the following questions mainly for $F$ the two-neighbor XOR automaton, although they are of interest for other surjective CA as well.

\begin{question}
Can we find a weakly mixing shift-invariant measure $\mu$  such that $F^n(\mu) \overset{n}{\longrightarrow} \delta_0$?
\end{question}

\begin{question}
Starting from an ergodic shift-invariant measure $\mu$, what can be said about the set of weak-* limit points of $F^n(\mu)$?
\end{question}

For the previous question, even the case where $\mu$ is a non-uniform Bernoulli measure is still largely open. Since measure-theoretical entropy is upper semicontinuous with respect to weak-* convergence, any weak-* limit point must have positive entropy, so in particular $F^n\mu \overset{n \to \infty}{\longrightarrow} \delta_0$ is not possible, but we do not know whether it is possible to find a word $w$ such that $(F^n\mu)([w]) \overset{n \to \infty}{\longrightarrow} 0$.

For a computability point of view,  we can see that any complexity in the limit measure in Theorem~\ref{thm:XORall1} was present in the initial measure and the CA itself is not doing any significant computation. Is this a restriction caused by surjectivity?

\begin{question}
Starting from a computable shift-invariant measure $\mu$, is it possible to have an uncomputable weak*-limit point for $F^n(\mu)$? Alternatively, is it possible that the $\mu$-limit set $\{ w \in \Sigma^\ast \mid F^n([w])\centernot\longrightarrow 0\}$ is uncomputable?
\end{question}

Based on computer searches, it seems likely that the gap between Theorem~\ref{thm:OneDomination} and Theorem~\ref{thm:HighDomination} can be closed, at least in the binary case.

\begin{conjecture}[Domination of prefix sums]
  For a surjective binary CA $F$ with radius $r$ and $0 \leq n \leq r+1$, we have
  \[ 
    \sum_{k = 0}^n N^1_1(F,r,k) \geq \sum_{k = 0}^n N^1_1(\ID,r,k).
  \]
\end{conjecture}

This would imply $C^1_1(F,r,m) \leq C^1_1(\ID,r,m)$ for all positive integers $m$.

\begin{credits}
\subsubsection{\ackname} This work started during a stay of the first author at Turku University, and Benjamin Hellouin is grateful to Jarkko Kari for the invitation (and the orienteering). Ilkka Törmä was supported by the Academy of Finland under grant 359921.

\subsubsection{\discintname}
The authors have no competing interests to declare that are relevant to the content of this article.
\end{credits}

\bibliographystyle{splncs04}
\bibliography{discrepancybib}
\end{document}